\documentclass[a4paper, 11pt]{amsart}

\usepackage{amsmath,amssymb,amsthm,amscd,enumerate,setspace}
\usepackage[all]{xy}
\usepackage[dvips]{graphicx}

\usepackage{xcolor}
\usepackage{enumitem}
\usepackage{mdframed}

\newmdenv[backgroundcolor=yellow]{shaded}

\usepackage{hyperref}
\newtheorem{Theorem}{Theorem}[section]
\newtheorem{Lemma}[Theorem]{Lemma}
\newtheorem{Corollary}[Theorem]{Corollary}
\newtheorem{Proposition}[Theorem]{Proposition}

\newtheorem{Remark}[Theorem]{Remark}
\newtheorem{Example}[Theorem]{Example}
\newtheorem{Definition}[Theorem]{Definition}
\newtheorem{Question}[Theorem]{Question}

\newtheorem{claim}[Theorem]{Claim}

\def\ann{\mbox{\rm ann}}
\def\Ass{\mbox{\rm Ass}}

\def\C{\mathcal{C}}
\def\cdeg{\mbox{\rm cdeg}}

\def\ds{\displaystyle}

\def\Ext{\mbox{\rm Ext}}

\def\gr{\mbox{\rm gr}}

\def\height{\mbox{\rm height }}
\def\Hom{\mbox{\rm Hom}}

\def\ker{\mbox{\rm ker}}

\def\l{{\lambda}}
\def\lar{\longrightarrow}

\def\Min{\mbox{\rm Min}}

\def\rar{\rightarrow}
\def\red{\mbox{\rm red}}

\def\Spec{\mbox{\rm Spec}}

\def\e{\mathrm{e}}
\def\m{{\mathfrak m}}
\def\n{\mathfrak{n}}
\def\p{{\mathfrak p}}
\def\q{\mathfrak{q}}

\def\AA{{\mathbf A}}

\def\BB{{\mathbf B}}

\def\RR{{\mathbf R}}

\def\xx{{\mathbf x}}
\def\TT{{\mathbf T}}

\def\g2{{\mathbf g}}

\def\xx{{\mathbf x}}

\def\yy{{\mathbf y}}

\def\H{{\mathrm H}}

\newcommand{\rmK}{\mathrm{K}}
\newcommand{\rmQ}{\mathrm{Q}}

\begin{document}

\title{Invariants of Cohen-Macaulay rings associated to their
canonical ideals}

\author{L. Ghezzi} \address{Department of Mathematics, New York City
College of Technology-Cuny, 300 Jay Street, Brooklyn, NY 11201, U.S.A.} \email{lghezzi@citytech.cuny.edu}
\author{S. Goto}
\address{Department of Mathematics, School of Science and Technology,
Meiji University, 1-1-1 Higashi-mita, Tama-ku, Kawasaki 214-8571,
Japan} \email{goto@math.meiji.ac.jp}
\author{J. Hong}
\address{Department of Mathematics, Southern Connecticut State
University, 501 Crescent Street, New Haven, CT 06515-1533, U.S.A.}
\email{hongj2@southernct.edu}
\author{W. V. Vasconcelos}
\address{Department of Mathematics, Rutgers University, 110
Frelinghuysen Rd, Piscataway, NJ 08854-8019, U.S.A.}
\email{wvasconce@gmail.com}

\thanks{The first author was partially supported by a grant from the City University of New York PSC-CUNY Research Award Program-46. }

\begin{abstract}
The purpose  of this paper  is to introduce new invariants of Cohen-Macaulay local rings.
Our focus is the class of Cohen-Macaulay local rings that admit a canonical ideal. Attached
to each such ring $\RR$ with a canonical ideal $\C$, there are
  integers--the type of $\RR$, the reduction number of $\C$--that provide valuable metrics
  to express the deviation of $\RR$ from being a Gorenstein ring.
We  enlarge  this list with other integers--the roots of $\RR$ and several canonical degrees. The
latter are multiplicity based functions of the Rees algebra of $\C$.
%The
 %main effort is  setting the foundations
 %of the new invariants and examining their relationships. In a sequel we make applications to Rees algebras,    monomial rings, Stanley-Reisner rings.
\end{abstract}

%{\today}

\maketitle

\noindent {\small {\bf  Key Words and Phrases:}
Canonical degree, Cohen-Macaulay type, analytic spread, roots, reduction number.}

%\tableofcontents

\section{Introduction}

%{\bf This is very preliminary. Just moved some material to seek a consistent narrative.}

%\medskip

\noindent
Let $\RR$ be a Cohen-Macaulay  ring of dimension $d \geq 1$.
% with an infinite residue field.
 If $\RR$ admits a canonical module $\C$ and has a Gorenstein total ring of fractions, we may assume that   $\C$ is
 an  ideal  of $\RR$. In this case,
   we  introduce  new numerical
   invariants for $\RR$ that refine and extend
  for local rings  the use of its Cohen-Macaulay {\em type}, the minimal number of generators
   of $\C$, $r(\RR) = \nu(\C)$.
   Among numerical invariants of
   the isomorphism class of $\C$ are the {\em analytic spread} $\ell(\C)$ of $\C$ and attached  reduction numbers. We also introduce  the {\em rootset} of $\RR$, which may be a novel invariant.
   Certain constructions on $\C$, such as the Rees algebra $\RR[\C\TT]$ leads to an invariant
 of $\RR$,
but the associated graded ring $\gr_{\C}(\RR)$ does not. It carries however properties of a semi-invariant
which we will make use of to  build true invariants. Combinations of semi-invariants are then used
to build invariants under the general designation of {\em canonical degrees}.
The
 main effort is  setting the foundations of the new invariants and examining their relationships.
    We shall also experiment in extending the construction to more general rings.
  When we do so, to facilitate the discussion we assume that $\RR$ is a homomorphic image of a Gorenstein ring.
 In a sequel we make applications to Rees algebras, monomial rings, Stanley-Reisner rings.

\medskip

%\begin{enumerate}

%\item
In Section 2 we introduce our basic canonical degree and derive some of its most direct properties. It requires knowledge of the Hilbert
coefficients  $\e_0(\cdot)$
of $\m$-primary ideals:
\medskip

\noindent
{\bf Theorem \ref{gencdeg1}.}
Let $(\RR, \m)$ be a Cohen-Macaulay local ring  of dimension $d \geq 1$  that has a canonical ideal $\C$.
  Then
  \[ \cdeg(\RR)= \sum_{\tiny \height \p=1} \cdeg(\RR_{\p}) \deg(\RR/\p) = \sum_{\tiny \height \p=1} [\e_{0}(\C_{\p}) - \l((\RR/\C)_{\p})] \deg(\RR/\p)
  \] is a well-defined finite sum independent of the chosen canonical ideal $\C$. In particular, if $\C$ is equimultiple with a minimal reduction $(a)$, then
  \[  \cdeg(\RR) = \deg(\C/(a)) = \e_0(\m, \C/(a)).\]

Two of its consequences when $\C$ is equimultiple are: (i) $\cdeg(\RR) \geq r(\RR) - 1$; (ii) $\cdeg(\RR) = 0$ if and only if $\RR$ is Gorenstein.
\medskip

%\item
In Section 3,
the minimal value for $\cdeg(\cdot)$ is assumed  on a new class of Cohen-Macaulay rings, called {\em almost Gorenstein rings}, introduced
in \cite{BF97}, and developed in \cite{GMP11} and \cite{GTT15}:
\medskip

\noindent
{\bf Proposition \ref{almostg}.}
Let $(\RR, \m)$ be a Cohen-Macaulay local ring  of dimension $d \geq 1$  that has a canonical ideal $\C$. If $\C$ is equimultiple and
$\cdeg(\RR) = r(\RR) - 1$, then there is an exact sequence
of $\RR$-modules
${\ds 0\rightarrow\RR\rightarrow \C \rightarrow X\rightarrow  0}$
such that $\nu(X)=\e_0(X)$.

\medskip

%\item
The next three sections
 carry out several general calculations seeking relations between $\cdeg(\RR)$ and other invariants of $\RR$ in the case of low
 Cohen-Macaulay type. For instance, it
 treats the notion of the {\em rootset} of $\RR$,  made up  of
 the positive integers $n$ such that $L^n \simeq \C$ for some
fractional ideal $L$. In dimension $1$ this is a finite set of cardinality $\leq r(\RR) - 1$.  We also describe the rootset for fairly general monomial rings.
 Section 6 deals with general properties of $\cdeg(\cdot)$ under a change of rings: polynomial rings,  completion and hyperplane section.  It has an extended  but quick treatment of the changes of $\cdeg(\AA)$ when $\AA$ is the ring obtained by augmenting $\RR$ by a module.

\medskip

%\item
In Section 7 we discuss extensions of the canonical degree. A clear shortcoming of the definition of $\cdeg(\RR)$, e.g. that be independent of the canonical ideal and that its vanishing is equivalent to
the Gorenstein property, is that it does not seem to work for rings of dimension $\geq 2$. A natural choice would be for
  $G = \gr_{\C}(\RR) $ to define
 the {\em canonical degree}$^*$ of $\RR$  as  the integer
$ \cdeg^{*}_{\C}(\RR)
 =\deg(\gr_{\C}(\RR)) - \deg(\RR/\C)
$.
Finally, an appropriate
hyperplane section readily creates a degree that has the Gorenstein property but we were unable to prove the independence
property.
%%-------------------

%\medskip

%\end{enumerate}

 % For basic references, we will use \cite{BS} and \cite{BH}.

%\begin{shaded}
%\begin{Notation}{\rm Let $(\RR, \m)$ be a Cohen-Macaulay local ring with infinite residue field.
%\begin{enumerate}
%\item[(1)] $\ell(\tratto) =$ analytic spread.
%\item[(2)] $\l(\tratto)= $ length.
%\item[(3)] $\nu(\tratto)=$ minimal number of generators.
%\item[(4)] $\omega=$  canonical module of $\RR$.
%\item[(5)] $\C=$ canonical ideal of $\RR$.
%\item[(6)] $\red_Q (I)=$ reduction number of ideal $I$ with respect to a reduction $Q$.
%\item[(7)] $r(\RR)= \nu(\omega)= $ type of the ring $\RR$.
%\item[(8)] $\mbox{\rm root}(\RR) = $ root of $\RR$.
%\item[(9)] $\deg(M) = \e_0(\m;M) =$ ordinary multiplicity of the $\RR$-module $M$.
%\item[(10)] $\cdeg(\RR) = $ basic canonical degree of $\RR$.
%\item[(11)] $\cdeg^*(\RR) = $ variation of the canonical degree of $\RR$.
%\item[(12)] $\cdeg_C(\RR) = $  relative  canonical degree of $\RR$.
%\item[(13)] $\rho(\RR)= \red(\C)$ canonical index of $\RR$.
%\end{enumerate}
%}\end{Notation}
%\end{shaded}

\section{Canonical degree}

%%--------------------------------------
%Let $(\RR,\m)$ be a Cohen-Macaulay local ring with an infinite residue field. Suppose that $\RR$  has a canonical ideal $\C$. In this setting we associate to $\C$ several metrics for the ring $\RR$.
%
%\begin{itemize}
%\item The {\em type} of $\RR$: $r(\RR) = \nu(\C)$.
%\item The {\em canonical index} of $\RR$:  $\rho(\RR)= \red(\C)$, that is  the smallest integer $n$ such that $\C^{n+1} = P\C^n$ taken over  all minimal reductions $P$ of $\C$.
%\item A {\em root} of $\C $ is an ideal $L$ such that $L^n \simeq \C$, for some integer $n$.  By abuse of terminology, we define ${\mathrm root}(\RR) =  \{ n \mid L^{n} \simeq	\C \;\; \mbox{for some ideal} \; L \}$.
%\end{itemize}
%
%All of these numbers are independent of $\C$, for which reason we refer to them as invariants of $\RR$. Still others, such as $\ell(\C)$, arise when we examine the isomorphisms of the $\RR$-algebras  $\RR[\C \TT]$.
%%------------------------------------------

%\subsubsection*{The basic canonical degree of a ring}
\noindent
  For basic references on canonical modules we  will use
\cite{BH},   \cite{BS} and \cite{HK2}, while for the existence and properties of canonical ideals we
   use \cite{Aoyama}.
Let $(\RR,\m)$ be a Cohen-Macaulay local ring. Suppose that $\RR$  has a canonical ideal $\C$.  In this setting we
 introduce a numerical degree for $\RR$ and study its properties. The starting point of our discussion is the
following elementary observation. We denote the length by $\lambda$.
%%------------------
%In order to have a more formal presentation,  we going to make numerous observations about canonical ideals, their reductions and some of their divisorial properties.
%%-------------------

 \begin{Proposition}\label{cdeg}
 Let $(\RR,\m)$ be a $1$-dimensional Cohen-Macaulay local ring with  a canonical ideal $\C$.
 Then the integer $\cdeg(\RR)=\e_0(\C) -\l(\RR/\C)$ is independent of the canonical ideal $\C$.
 \end{Proposition}

\begin{proof}  If $x$ is an indeterminate over $\RR$, in calculating these differences we may pass from $\RR$
to $\RR(x) = \RR[x]_{\m \RR[x]}$, in particular we may assume that the ring has an infinite residue field.

Let $\C$ and $\mathcal{D}$ be two canonical ideals. Suppose $(a)$ is a minimal reduction of $\C$.
Since $\mathcal{D} \simeq \C$ (\cite[Theorem 3.3.4]{BH}), $\mathcal{D} = q \C$ for some fraction $q$.  If $\C^{n+1} = (a) \C^n$
 by multiplying it by $q^{n+1}$,  we get $\mathcal{D}^{n+1} =(qa) \mathcal{D}^n$, where $(qa) \subset \mathcal{D}$. Thus $(qa)$ is a reduction of $\mathcal{D}$ and
$\C/(a) \simeq \mathcal{D}/(qa)$. Taking their co-lengths we have
\[ \l(\RR/(a)) - \l(\RR/ \C) = \l(\RR/(qa)) - \l(\RR/ \mathcal{D}).\]
  Since $\l(\RR/(a)) = \e_0(\C)$ and $\l(\RR/(qa)) = \e_0(\mathcal{D})
 $,
we have
\[ \e_0(\C) - \l(\RR/ \C) = \e_0( \mathcal{D}) - \l(\RR/ \mathcal{D}).\]
  \end{proof}

 We can define $\cdeg(\RR)$ in full generality as follows.

 \begin{Theorem}\label{gencdeg1} Let $(\RR, \m)$ be a Cohen-Macaulay local ring  of dimension $d \geq 1$  that has a canonical ideal $\C$.
  Then
  \[ \cdeg(\RR)= \sum_{\tiny \height \p=1} \cdeg(\RR_{\p}) \deg(\RR/\p) = \sum_{\tiny \height \p=1} [\e_{0}(\C_{\p}) - \l((\RR/\C)_{\p})] \deg(\RR/\p)
  \] is a well-defined finite sum independent of the chosen canonical ideal $\C$. In particular, if $\C$ is equimultiple with a minimal reduction $(a)$, then
  \[  \cdeg(\RR) = \deg(\C/(a)) = \e_0(\m, \C/(a)).\]
 \end{Theorem}

\begin{proof} By Proposition \ref{cdeg}, the integer $\cdeg(\RR_{\p})$ does not depend on the choice of a canonical ideal of $\RR$. Also $\cdeg(\RR)$ is a finite sum since, if $\p \notin \Min(\C)$, then $\C_{\p}=\RR_{\p}$ so that $\RR_{\p}$ is Gorenstein. Thus $\cdeg(\RR_{\p})=0$. The last assertion follows from the associativity formula:
\[ \cdeg(\RR) =  \sum_{\tiny \height \p=1} \l( (\C/(a))_{\p}) \deg(\RR/\p) =  \deg(\C/(a)).\]
\end{proof}

\begin{Definition}\label{defcdeg}{\rm
Let $(\RR, \m)$ be a Cohen-Macaulay local ring  of dimension $d \geq 1$ that has a canonical ideal.
  Then the {\em canonical degree of $\RR$} is the integer
  \[ \cdeg(\RR)= \sum_{\tiny \height \p=1} \cdeg(\RR_{\p}) \deg(\RR/\p).
  \]
  }\end{Definition}

\begin{Corollary}
%Let $(\RR, \m)$ be a Cohen-Macaulay local ring  of dimension $d \geq 1$ that has a canonical ideal $\C$.
$\cdeg(\RR)\geq 0$ and vanishes if and only if  $\RR$ is Gorenstein in codimension $1$.
% Furthermore, suppose that $\C$ is equimultiple. Then $\cdeg(\RR) =0$ if and only if $\RR$ is Gorenstein.
\end{Corollary}

%\begin{proof} Suppose that $(a) \subset \C$ is a minimal reduction. Then, from the exact sequence
%\[ 0 \rar \C/(a) \rar \RR/(a) \rar \RR/\C \rar 0,\]
%where $\RR/(a)$ and $\RR/\C$ are Cohen-Macaulay of dimension $d-1$. Then either $\C/(a) =(0)$ or $\dim(\C/(a))=d-1$. Since $ \cdeg(\RR)= \deg(\C/(a))= \deg(\RR(a)) - \deg(\RR/\C) =0$,   we have $\C/(a) = (0)$.
%\end{proof}

\begin{Corollary}\label{cdegr1}
%Let $(\RR, \m)$ be a Cohen-Macaulay local ring   that   has a canonical ideal $\C$.
Suppose that the canonical ideal of $\RR$ is equimultiple. Then we have the following.
\begin{enumerate}
\item[{\rm (1)}] $\cdeg(\RR) \geq r(\RR)-1$.
\item[{\rm (2)}] $\cdeg(\RR) =0$ if and only if $\RR$ is Gorenstein.
\end{enumerate}
\end{Corollary}

\begin{proof}
Let $(a)$ be a minimal reduction of the canonical ideal $\C$. Then
\[ \cdeg(\RR) = \e_0(\m, \C/(a))\geq \nu(\C/(a))=r(\RR)-1.\]
If $\cdeg(\RR)=0$ then $r(\RR)=1$, which proves that $\RR$ is Gorenstein.
\end{proof}

 Now we extend the above result to a more general class of ideals when the ring has dimension one. We recall that if $\RR$ is a $1$-dimensional
 Cohen-Macaulay local ring, $I$ is an $\m$-primary ideal with  minimal reduction $(a)$,
 then the reduction number of $I$ relative to $(a)$ is independent of the reduction
 (\cite[Theorem 1.2]{Trung87}). It will be denoted simply by $\red(I)$.
%\noindent Now we extend this result to a more general class of ideals.

\begin{Proposition} Let $(\RR, \m)$ be a $1$-dimensional  Cohen-Macaulay local ring which is not a valuation ring. Let  $I$ be an irreducible $\m$-primary ideal such that $I \subset \m^2$.  If $(a)$
is a minimal reduction of $I$, then $\lambda(I/(a)) \geq r(\RR)-1$. In the case of equality, $\red(I) \leq 2$.
\end{Proposition}

\begin{proof}
Let $L = I:\m$ and $N=(a): \m$. Then ${\ds \l(L/ I) = r(\RR/ I) = 1}$ and ${\ds \l( N / (a) ) = r(\RR)}$. Thus, we have
\[ r(\RR) \leq  \l(L/ N )  +  \l( N / (a) ) = \l(L/(a) ) = \l(L/ I)  + \l( I/(a) ) = 1 + \l (I/(a) ),\]
which proves that $\lambda(I/(a)) \geq r(\RR)-1$.

\medskip

\noindent Suppose that $\lambda(I/(a)) = r(\RR)-1$. Then $L=N$. By \cite[Lemma 3.6]{CHV}, $L$ is integral over $I$. Thus, $L$ is integral over $(a)$. By \cite[Theorem 2.3]{CP95}, $\red(N)=1$. Hence $L^2 = a L$.  Since $\lambda(L/ I) = 1 $, by
\cite[Proposition 2.6]{GNO}, we have $I^3 = a I^2$.
\end{proof}

\section{Extremal values of the canonical degree}
\noindent
We examine in this section extremal values of the canonical degree.
%\subsubsection*{Almost Gorenstein rings}
First we recall the definition of almost Gorenstein rings  (\cite{BF97, GMP11, GTT15}).

\begin{Definition}{\rm (\cite[Definition 3.3]{GTT15})  A Cohen-Macaulay local ring $\RR$ with a canonical module $K_{\RR}$  is said to be an {\em almost Gorenstein} ring if there exists an exact sequence of $\RR$-modules
${\ds 0\rightarrow\RR\rightarrow K_{\RR} \rightarrow X\rightarrow  0}$
such that $\nu(X)=\e_0(X)$.
}\end{Definition}

%Let $Q=(a)$ be a minimal reduction of $\C$.
%Furthermore, if $d=1$ and $\RR$ is not a Gorenstein ring, we then have  $\red_Q(\C) = 2$.
%The last assertion follows from \cite[Theorem 3.16]{GMP11}.

\begin{Proposition}\label{almostg}
Let $(\RR,\m)$ be a Cohen-Macaulay local ring with a canonical ideal $\C$. Assume that $\C$ is equimultiple.
If $\cdeg (\RR) = r(\RR) -1$, then $\RR$ is an almost Gorenstein ring. In particular,  if $\cdeg (\RR) \le 1$, then $\RR$ is an almost Gorenstein ring.
\end{Proposition}

\begin{proof}
We may assume that $\RR$ is not a Gorenstein ring. Let $(a)$ be a minimal reduction of $\C$.
Consider the exact sequence of $\RR$-modules
\[ 0 \to \RR \overset{\varphi}{\to} \C \to X \to 0, \;\; \mbox{where} \;\; \varphi (1) = a. \] Then $\nu(X)=r(\RR) -1=\cdeg (\RR)=\e_0(X)$. Thus, $\RR$ is an almost Gorenstein ring.
\end{proof}

\begin{Proposition}\label{1dimalmostg}
Let $(\RR,\m)$ be a $1$-dimensional Cohen-Macaulay local ring with a canonical ideal $\C$. Then $\cdeg(\RR)= r(\RR)-1$ if and only if $\RR$ is an almost Gorenstein ring.
\end{Proposition}

\begin{proof}
It is enough to prove that the converse holds true. We may assume that $\RR$ is not a Gorenstein ring.  Let $(a)$ be a minimal reduction of $\C$.
Since $\RR$ is almost Gorenstein, there exists an exact sequence of $\RR$-modules
\[ 0 \to \RR \overset{\psi}{\to} \C \to Y \to 0 \;\; \mbox{such that} \;\; \nu(Y)=\e_0(Y).\]  Since $\dim(Y)=0$ by
\cite[Lemma 3.1]{GTT15}, we have that $\m Y=(0)$.
Let $b = \psi (1) \in \C$ and set $\q = (b)$. Then $\m \q \subseteq \m \C \subseteq \q$. Therefore, since $\RR$ is not a DVR and $\l(\q/\m \q) = 1$, we get $\m \C = \m  \q$, whence the ideal $\q $ is a reduction of $\C$ by the Cayley-Hamilton theorem, so that $\cdeg (\RR) = \e_0(Y) = \nu(Y) =r(\RR) - 1$.
\end{proof}

%%-----------------------
%\begin{Remark} \rm{I don't know whether the converse of Assertion (1) in Theorem \ref{almostg} holds true when $d \ge 2$, provided $\ell (\C) = 1$.}\end{Remark}
%%------------------------

Now we consider the general case when a canonical ideal is not necessarily equimultiple.

\begin{Lemma}\label{prop1}
Let $(\RR, \m)$ be a Cohen-Macaulay local ring of dimension $d \geq 1$ with infinite residue field and a canonical ideal $\C$.
Let $a$ be an element of  $\C$ such that {\rm (i)} for $\forall \p \in \Ass (\RR/\C)$ the element $\frac{a}{1}$ generates a reduction of $\C R_\p$, {\rm (ii)} $a$ is $\RR$-regular, and {\rm (iii)} $a \not\in \m \C$.  Let ${\ds Z=\{ \p \in \Ass(\C/(a)) \mid \C \not\subseteq \p \} }$.

\begin{enumerate}
\item[{\rm (1)}] ${\ds \cdeg (\RR) = \deg(\C/(a)) - \sum_{\p \in Z}\lambda( (\C/(a))_\p) \deg(\RR/\p)}$.

\item[{\rm (2)}]  $\cdeg (\RR) = \deg(\C/(a))$ if and only if $\Ass(\C/(a)) \subseteq V(\C)$.
\end{enumerate}
\end{Lemma}

\begin{proof}
It follows from Theorem \ref{gencdeg1} and  ${\ds \deg(\C/(a)) = \sum_{\tiny \height \p=1} \l((\C/(a))_\p)\deg(\RR/\p)}$.
\end{proof}

%Shiro: I am not sure whether we can remove in Theorem \ref{almostg2} the assumption that $\Ass_\RR X \subseteq V(\C)$.

\begin{Theorem}\label{almostg2} With the same notation given in Lemma {\rm \ref{prop1}}, suppose that $\Ass(\C/(a)) \subseteq V(\C)$. Then  $\cdeg (\RR) =r(\RR) - 1$ if and only if  $\RR$ is an almost Gorenstein ring.
\end{Theorem}

\begin{proof}
It is enough to prove the converse holds true. We may assume that $\RR$ is not a Gorenstein ring. Choose an exact sequence
\[ 0 \to \RR \to \C \to Y \to 0 \]
such that $\deg(Y)= \nu(Y) = r(\RR) - 1$. Since $\Ass(\C/(a)) \subseteq V(\C)$, we have $\deg(Y) \geq \deg(\C/(a))$.  By Lemma \ref{prop1} and the proof of Corollary \ref{cdegr1}, we obtain the following:
\[ r(\RR) -1 \leq \cdeg(\RR) = \deg(C/(a)) \leq \deg(Y) = r(\RR)-1.\]
\end{proof}

\begin{Remark}{\rm If $\RR$ is a non-Gorenstein normal domain, then its canonical ideal cannot be equimultiple.}
\end{Remark}

\begin{proof}
 Suppose that a canonical ideal $\C$ of $\RR$ is equimultiple, i.e., $\C^{n+1}= a \C^{n}$.  Then we would have an equation $(n+1) [ \C] = n[\C]$ in its divisor class group. This means that $[\C]=[0]$. Thus, $\C \simeq \RR$. Hence $\C$ can not be equimultiple.
%  [in the example because it is generated by $2$ elements. Hence $\C$ is its own minimal reduction.
\end{proof}

\section{Canonical index}
\noindent
Throughout the section, let $(\RR, \m)$ be a Cohen-Macaulay local ring of dimension $d \geq 1$ with infinite residue field and suppose that a canonical ideal $\C$ exists.
%\subsubsection*{Elementary properties}
We begin by showing  that the reduction number of a canonical ideal of $\RR$ is an invariant of  the ring.

\begin{Proposition}\label{Ca1}
Let $\C$ and $\mathcal{D}$ be canonical ideals of $\RR$. Then $\red(\C)=\red(\mathcal{D})$.
\end{Proposition}

\begin{proof}
 Let $K$ be the total ring of quotients of $\RR$. Then there exists $q \in K$ such that $\mathcal{D}=q\C$. Let $r = \red(\C)$ and $J$ a minimal reduction of $\C$ with $\C^{r+1}=J \C^{r}$. Then
\[ \mathcal{D}^{r+1}= (q \C)^{r+1}= q^{r+1} (J \C^{r}) = (qJ)(q \C)^{r}= qJ \mathcal{D}^{r}\] so that $\red(\mathcal{D}) \leq \red(\C)$. Similarly, $\red(\C) \leq \red(\mathcal{D})$.
\end{proof}

\begin{Definition}{\rm
Let $(\RR, \m)$ be a Cohen-Macaulay local ring of dimension $d \geq 1$ with a canonical ideal $\C$. The {\em canonical index} of $\RR$ is the reduction number of the canonical ideal $\C$ of $\RR$ and is denoted by $\rho(\RR)$.
}\end{Definition}

\begin{Remark}\label{R}{\rm Suppose that $\RR$ is not Gorenstein.
The following are known facts.
\begin{enumerate}
\item[(1)] If the canonical ideal of $\RR$ is equimultiple, then $\rho(\RR) \neq 1$.

\item[(2)] If $\dim \RR=1$ and $\e_{0}(\m) =3$, then $\rho(\RR)=2$.

\item[(3)] If $\dim \RR=1$ and $\cdeg(\RR)=r(\RR) -1$, then $\rho(\RR)=2$.
  \end{enumerate}
}\end{Remark}

\begin{proof}
(1) Suppose that $\rho(\RR)=1$. Let $\C$ be a canonical ideal of $\RR$ with $\C^2 = a\C$. Then  $\C a^{-1} \subset \Hom(\C , \C) = \RR$ so that $\C = (a)$. This is a contradiction.

\medskip

\noindent (2) It follows from the fact that, if   $(\RR, \m)$ is a $1$--dimensional Cohen-Macaulay ring and $I$ an $\m$--primary ideal, then $\red(I) \leq \e_{0}(\m)-1$.

\medskip

\noindent (3) It follows from Proposition~\ref{1dimalmostg} and \cite[Theorem 3.16]{GMP11}.
\end{proof}

%%---------------------
%\begin{Question}{\rm Suppose $\dim \RR=1$ and let $(a)$ be a minimal reduction of $\C$. In general, $\lambda(\C/(a)) \geq \nu(\C)-1$. What is the upper bound for $\red(\C)$ when $\lambda(\C/(a)) = \nu(\C)=2$?}
%\end{Question}
%%-----------------------

\subsubsection*{Sally module}   We examine briefly the Sally module associated to
the canonical ideal $\C$ in rings of dimension $1$. Let $Q=(a)$
be a minimal reduction of  $\C$ and consider the exact sequence of finitely generated $\RR[Q\TT]$-modules
\[ 0 \rar \C \RR[Q\TT] \lar \C \RR[\C\TT] \lar S_Q(\C) \rar 0. \]
Then the Sally module  $S=S_Q(\C) = \bigoplus_{n\geq 1} \C^{n+1}/\C Q^{n}$ of $\C$ relative to $Q$ is Cohen-Macaulay and, by
 \cite[Theorem 2.1]{red}, we have
 \[\e_1(\C) = \cdeg(\RR) + \sum_{j=1}^{\rho(\RR)-1} \lambda(\C^{j+1}/a\C^j) = \sum_{j=0}^{\rho(\RR) -1} \lambda(\C^{j+1}/a\C^j). \]
%(i), (ii): Since $\depth \gr_{\C}(\RR) \geq d-1$, $S$ is Cohen-Macaulay.

\begin{Remark} {\rm Let $\RR$ be a $1$-dimensional Cohen-Macaulay local ring with a canonical ideal $\C$. Then the multiplicity of the Sally module $s_0(S)=\e_1(\C)-\e_0(\C)+\lambda(\RR/\C)=\e_1(\C)-\cdeg(\RR)$ is an invariant of the ring $\RR$, by \cite[Corollary 2.8]{GMP11} and Proposition \ref{cdeg}.}
\end{Remark}

The following property of Cohen-Macaulay rings of type $2$ is a useful calculation that we will use to characterize rings with minimal canonical index.

\begin{Proposition}\label{C2}
Let $\RR$ be a $1$-dimensional Cohen-Macaulay local ring with a canonical ideal $\C$.  Let $(a)$ be a minimal reduction of $\C$. If $\nu(\C)=2$, then $\l(\C^{2}/a \C)=\l(\C/(a))$.
\end{Proposition}

\begin{proof}
Let $\C=(a, b)$ and consider the exact sequence
\[ 0 \rar Z \rar \RR^{2} \rar \C \rar 0,\]
where $Z=\{ (r,s) \in \RR^{2} \mid ra+sb =0 \}$.  By tensoring this exact sequence with $\RR/ \C$, we obtain
\[ Z/\C Z  \stackrel{g}{\rar} (\RR/\C)^{2} \stackrel{h}{\rar} \C/\C^{2} \rar 0.\]
Then we have
\[ \ker(h)= \mbox{Im}(g) \simeq (Z/ \C Z)/ ( (Z \cap \C \RR^{2})/ \C Z ) \simeq Z/(Z \cap \C \RR^{2}) \simeq  (Z/B)/ ( (Z \cap \C \RR^{2})/B ),\]
where $B= \{ (-bx, ax) \mid x \in \RR \}$.

\medskip

\noindent We claim that $Z \cap \C \RR^{2} \subset B$, i.e., $\delta(\C)= (Z \cap \C \RR^{2})/B =0$. Let $(r, s) \in Z \cap \C \RR^{2}$. Then
\[ ra + sb =0 \;\; \Rightarrow \;\;  \frac{s}{a} \cdot b = -r \in \C \; \mbox{and} \; \frac{s}{a} \cdot a = s \in \C.\]
Denote the total ring of fractions of $\RR$ by $K$. Since $\C$ is a canonical ideal, we have
\[ \frac{s}{a} \in \C :_{K} \C =\RR.\]
Therefore
\[ (r, s) = \left( -b \cdot \frac{s}{a}, \; a \cdot \frac{s}{a} \right) \in B.\]
Hence $\ker(h)  \simeq Z/B =\H_{1}(\C)$ and we obtain the following exact sequence
\[ 0 \rar \H_{1}(\C) \rar (\RR/ \C)^{2} \rar \C/ \C^{2} \rar 0.\]
Next we claim that $\l(\H_{1}(\C))= \l(\RR/ \C)$.  Note that $\H_{1}(\C) \simeq ((a):b)/(a)$ by  mapping $(r,s)+B$ with $ra+sb=0$ to $s+(a)$. Using the exact sequence
\[ 0 \rar ((a):b)/(a) \rar \RR/(a) \stackrel{\cdot b}{\rar}  \RR/(a) \rar  \RR/\C \rar 0,\]
 we get
 \[ \l(\RR/\C) =  \l( ((a):b)/(a) ) = \l(\H_{1}(\C)).\]
Now, using the exact sequence
\[ 0 \rar \H_{1}(\C) \rar (\RR/\C)^{2} \rar \C/\C^{2} \rar 0, \]
we get
\[ \l(\C/\C^{2}) = 2 \l(\RR/\C) - \l(\H_{1}(\C)) = \l(\RR/\C).\]
Hence,
\[ \l(\C^{2}/a \C) = \l(\C/a \C) - \l(\C/ \C^{2}) = \l(\C/a \C) - \l(\RR/ \C) = \l(\C/a \C) - \l((a)/a \C) = \l(\C/(a)).\]
\end{proof}

\begin{Theorem}\label{SallyofC} Let $ (\RR, \m)$ be a $1$-dimensional Cohen-Macaulay local ring with a canonical ideal $\C$. Suppose that the type of $\RR$ is $2$.
Then we have the following.
\begin{itemize}
\item[{\rm (1)}] ${\ds \e_{1}(\C) \leq \rho(\RR) \, \cdeg(\RR)}$.
\item[{\rm (2)}] $\rho(\RR)=2$ if and only if ${\ds \e_{1}(\C) = 2 \, \cdeg(\RR)}$.
\end{itemize}
\end{Theorem}

\begin{proof} Let $\C = (a, b)$, where $(a)$ is a minimal reduction of $\C$.

\medskip

\noindent (1) For each $j=0, \ldots, \rho(\RR) -1$,  the module $\C^{j+1}/a\C^{j}$ is cyclic and annihilated by $L=\ann(\C/(a))$ . Hence we obtain
\[\e_1(\C) = \sum_{j=0}^{\rho(\RR) -1} \lambda(\C^{j+1}/a\C^j)  \leq  \rho(\RR) \, \lambda(\RR/L)= \rho(\RR)  \, \cdeg(\RR).\]
(2) Note that $\rho(\RR)=2$ if and only if ${\ds \e_{1}(\C) = \sum_{j=0}^{1} \lambda(\C^{j+1}/a\C^j) }$.
Since $\nu(\C)= r(\RR)=2$, by Proposition~\ref{C2}, $\l(\C/(a) ) = \l(\C^{2}/a \C)$.  Thus, the assertion follows from
\[  \sum_{j=0}^{1} \lambda(\C^{j+1}/a\C^j) = 2 \, \l(C/(a)) = 2 \, \cdeg(\RR).\]
\end{proof}

\medskip

\begin{Example}{\rm Let $H=\left<a,b,c \right>$ be a numerical semigroup which is minimally generated by positive integers $a,b,c$ with gcd$(a,b,c)=1$. If the semigroup ring $\RR=k[[t^a,t^b,t^c]]$ is not a Gorenstein ring, then $r(\RR)=2$ (see \cite[Section 4]{GMP11}).}
\end{Example}

\begin{Example}\label{Ex1linkage}{\rm Let $A=k[X, Y, Z]$, let
$I=(X^2-YZ, Y^2-XZ, Z^2-XY)$ and $\RR=A/I$. Let $x, y, z$ be the images of $X, Y, Z$ in $\RR$. By \cite[Theorem 10.6.5]{ABAbook}, we see that $\C=(x, z)$ is a canonical ideal of $\RR$ with a minimal reduction $(x)$. It is easy to see that $\rho(\RR)=2$, $\e_1(\C)=2$ and $\cdeg(\RR)=1$.}
\end{Example}

\begin{Example}{\rm Let $A=k[X,Y,Z]$, let $I=(X^4-Y^{2}Z^{2}, Y^{4}-X^{2}Z^{2}, Z^{4}-X^{2}Y^{2})$ and $\RR=A/I$. Let $x, y, z$ be the images of $X, Y, Z$ in $\RR$. Then $\C=(x^{2}, z^{2})$ is a canonical ideal of $\RR$ with a minimal reduction $(x^{2})$. We have that $\rho(\RR)=2$, $\e_1(\C)=16$ and $\cdeg(\RR)=8$.}
\end{Example}

\medskip

\subsubsection*{Lower and upper bounds for the canonical index}

\begin{Example} {\rm Let $e \ge 4$ be an integer and let $H = \left<e, \{e+i\}_{3 \le i \le e-1}, 3e+1, 3e+2\right>$.
Let $k$ be a field and $V = k[[t]]$  the formal power series ring over $k$. Consider the semigroup ring $\RR = k[[H]] \subseteq V$.
\begin{enumerate}
\item[(1)] The conductor of $H$ is  $c= 2e+3$.
\item[(2)] The canonical module is $K_{\RR} = \left<1, t\right>$ and $K_{\RR}^{e-2} \subsetneq K_{\RR}^{e-1} = V$.
\item[(3)] The  canonical ideal of $\RR$ is $\C=(t^c K_{\RR})$ and $Q = (t^c)$ is a minimal reduction of $\C$. Moreover, $\rho(\RR)= \red(\C)= e-1$.
\item[(4)] The canonical degree is $\cdeg(\RR)= \lambda(\C/Q) = \lambda(K_{\RR}/\RR) = 3$.
\item[(5)] In particular, $\cdeg(\RR) \leq \rho(\RR)$.
\end{enumerate}
 }\end{Example}

 \begin{Proposition}
 Let $(\RR, \m)$ be a Cohen-Macaulay local ring with infinite residue field and a canonical ideal $\C$. Suppose that $\RR_\p$ is a Gorenstein ring for $\forall \p \in \Spec (\RR) \setminus \{\m\}$ and that $\C$ is equimultiple.
\begin{enumerate}
\item[{\rm (1)}] $\C^n$ has finite local cohomology for all $n > 0$.

\item[{\rm (2)}] Let $\Bbb I (\C^n)$ denote the Buchsbaum invariant of $\C^n$. Then the nonnegative integer
$ \beta(\RR) =  \sup_{n > 0}\Bbb I (\C^n) $ is independent of the choice of $\C$.

\item[{\rm (3)}] $\rho(\RR) \leq \deg(\RR) + \beta(\RR) -1$.
\end{enumerate}
 \end{Proposition}

 \begin{proof}
 (1) The assertion follows from $\C^n\RR_\p = a^n\RR_\p$ where $Q = (a)$ is a reduction of $\C$.

 \medskip

 \noindent (2)  We have $\C^n \cong \mathcal{D}^n$ for any canonical ideal $\mathcal{D}$ and $\C^{n+1} = a \C^n$ for $\forall n \gg 0$.

 \medskip

 \noindent (3) Let $\q$ be a minimal reduction of $\m$. Then, for $\forall n > 0$, we have
\[ \nu(\C^n) = \l(\C^n/\m \C^n) \leq \l(\C^n/\q \C^n) \leq \e_0(\q,\C^n) + \Bbb I (\C^n) \leq  \deg(\RR) + \beta(\RR). \]
The conclusion follows from \cite[Theorem 1]{ES}.
\end{proof}

%(From the Japanese Seminars)
\begin{Example}{\rm  Consider the following examples of $1$-dimensional Cohen-Macaulay semigroup rings $\RR$ with a canonical ideal $\C$ such that $\cdeg(\RR) =r(\RR)$.
\begin{enumerate}
\item[(1)] Let $a\geq 4$ be an integer and let $H= \left<a, a+3, \ldots, 2a-1, 2a+1, 2a+2 \right>$.  Let $\RR=k[[H]]$. Then the canonical module of $\RR$ is
$ K_{\RR} = \left< 1, t, t^3, t^4, \ldots, t^{a-1} \right>.$

%\[ K_{\RR} = \left< 0, 1, 3, 4, \ldots, a-1, a, a+1, a+3, \ldots \right>.\]
\noindent The ideal $Q=(t^{a+3})$  is a minimal reduction of $\C= (t^{a+3}, t^{a+4}, t^{a+6}, t^{a+7}, \ldots, t^{2a+2} )$,
\[ \cdeg(\RR)=a-1=\nu(\C), \quad \mbox{and} \quad \red (\C)=2.\]

\item[(2)]  Let $a\geq 5$ be an integer and let $H= \left<a, a+1, a+4, \ldots, 2a-1, 2a+2, 2a+3 \right>$. Let $\RR=k[[H]]$. Then the canonical module of $\RR$ is
 $K_{\RR} = \left< 1, t, t^4, t^5, \ldots, t^{a-1} \right>$.

\noindent The ideal $Q=(t^{a+4})$ is a minimal reduction of
$\C= (t^{a+4}, t^{a+5}, t^{a+8}, t^{a+9}, \ldots, t^{2a+3})$.
\[ \cdeg(\RR)=a-2=\nu(\C), \quad \mbox{and} \quad \red(\C)=3.\]

%\item[(3)] Let $H=\left<a, a+1, a+4, \cdots, a+l, a+i+3, \ldots, 2a-1 \right>$. Let $\RR=k[[H]]$. Then
%\[ K_{\RR} = \left< 0, 1, l+3, l+4, \ldots, a-1, a, a+1, \ldots, a+l, a+l+1, a+l+3 \right>.\]
%The ideal $Q=(t^{a+l+3})$ is a minimal reduction of $\C=\left<t^{a+l+3}, t^{a+l+4}, t^{a+2l+6},  \ldots, t^{2a+l+2} \right>$.
%\[ \cdeg(\RR)=a-l-1=\nu(\C), \quad \mbox{and} \quad \red_{Q}(\C)=l+2.\]

\end{enumerate}
}\end{Example}

%\begin{Question}{\rm
%Let $\RR$ be an $1$-dimensional Cohen-Macaulay local ring with a canonical ideal. Suppose that $r(\RR)=\cdeg(\RR)$. Is it true that
%\[ \rho(\RR) = \deg(\RR)-\cdeg(\RR) +1? \]
%}\end{Question}

\section{Roots of canonical ideals}
\noindent
 Another phenomenon concerning canonical ideals  is the
 following that may impact the value of $\cdeg(\RR)$.

\begin{Definition}{\rm
Let $(\RR,\m)$ be a Cohen-Macaulay local ring of dimension $d \geq 1$ with a canonical ideal $\C$.
An ideal $L$ is called a {\em root} of $\C$ if $L^{n} \simeq \C$ for some $n$.
In this case, we write $\tau_{L}(\C) = \min \{ n \mid L^{n} \simeq \C \}$.
Then the {\em rootset} of $\RR$ is the set  $\mathrm{root}(\RR) = \{ \tau_{L}(\C) \mid L \; \mbox{is a root of} \; \C \} $.
} \end{Definition}

 The terminology  `roots' of $\C$ already appears in \cite{BV1}.
 %(This ref has some overstated statements, will point out.)
  Here is  a simple example.

 \begin{Example}{\rm (\cite[Example 3.4]{BV1})} \label{ex1root}
 {\rm Let $\RR=k[[t^4, t^5, t^6, t^7]]$. Then $\C= (t^4, t^5, t^6)$ is a canonical ideal of $\RR$.
 Let $I = (t^4, t^5)$. Then  $I^2 = t^4\C$, that is, $I$ is a square root of $\C$.
 }\end{Example}

 The set  $\mathrm{root}(\RR)$ is clearly independent of the chosen canonical ideal.
 To make clear the character of this set we appeal to the following property. We use the terminology
 of \cite{BV1} calling an ideal $L$ of $\RR$ {\em closed} if $\Hom(L,L)=\RR$.

   \begin{Proposition}\label{root0}
Let $(\RR,\m)$ be a $1$-dimensional Cohen-Macaulay local ring with a canonical ideal $\C$.
Let $L$ be a root of $\C$.
%\begin{enumerate}
%\item[{\rm (1)}]
 If $\Hom(L^{n}, L^{n}) \simeq \RR$ for infinitely many values of $n$ then
 all powers of $L$ are closed. In this case $L$ is invertible and $\RR $ is Gorenstein.
%
%\item[{\rm (2)}] If $L^m \simeq \C \simeq L^n$, for $m\neq n$, then $C$ is principal.
%\end{enumerate}
 \end{Proposition}

\begin{proof}
% We may assume that $\RR$ has an infinite residue field.
%\noindent (1)
 A property of roots is that they are closed ideals. More generally, it is clear that
  if $\Hom(L^m, L^m) = \RR$ then $\Hom(L^n, L^n)=\RR$ for $n<m$, which shows the first assertion.
%  as it  is  clear that if some
 %power of $L$ is not closed, then the following powers of $L$ are not closed either.

\medskip
%
%\noindent (2) Suppose $m > n$. Then
%\[ \C \simeq L^m = L^n L^{m-n} \simeq L^m L^{m-n} = L^{2m -n}.\] Iterating,
% $L$ is a root of $\C$ of arbitrarily high order.
%This is enough to show that  $L$ is principal. Let us give the standard argument.
 We may assume that $\RR$ has an infinite residue field.
Let $s=\red(L)$. Then $L^{s+1} = xL^s$, for a minimal reduction $(x)$ and thus
$ L^{2s} = x^sL^s$, which gives that
\[ x^{-s}L^s \subset \Hom(L^s,L^s) \simeq \RR.\]
It follows that $L^s \subset (x^s) \subset L^s$, and thus $L^s = (x^s)$. Now taking  $t$ such that $L^t\simeq C$ shows that $C$ is principal.
%and so $L$ is an invertible ideal.
\end{proof}

\begin{Corollary} \label{uniqueroot}
 If $L^m \simeq \C \simeq L^n$, for $m\neq n$, then $C$ is principal.
\end{Corollary}
\begin{proof}
 Suppose $m > n$. Then
\[ \C \simeq L^m = L^n L^{m-n} \simeq L^m L^{m-n} = L^{2m -n}.\] Iterating,
 $L$ is a root of $\C$ of arbitrarily high order.
 \end{proof}

\begin{Corollary}{\rm (\cite[Proposition 3.8]{BV1})}
Let $(\RR,\m)$ be a $1$-dimensional Cohen-Macaulay local ring with a canonical ideal $\C$.
If $\RR$ is not Gorenstein then no proper power of $\C$ is a canonical ideal.
 \end{Corollary}

 \begin{Proposition}\label{root1}
 Let $(\RR,\m)$ be a Cohen-Macaulay local ring of dimension $d \geq 1$ with infinite residue field and with an equimultiple canonical ideal $\C$.
 Let $L$ be a root of $\C$. Then  $\tau_{L}(\C)  \leq \min\{ r(\RR)-1, \red(L)\}$.
 \end{Proposition}

\begin{proof} Suppose   $n = \tau_{L}(\C) \geq r(\RR)$. Then
\[  \nu(L^{n}) = \nu (\C) = r(\RR) < n+1 =  {{n+1}\choose{1}}.\]
By \cite[Theorem 1]{ES}, there exists a reduction $(a)$ of $L$ such that
$L^{n} = aL^{n-1}$. Thus, $L^{n-1} \simeq \C$, which contradicts the minimality of $\tau_{L}(\C)$.
\end{proof}

\begin{Remark}{\rm  %Let $(\RR,\m)$ be a Cohen-Macaulay local ring of dimension $d \geq 1$ with a canonical ideal $\C$.
We have the following.

\begin{enumerate}
\item[(1)] If $r(\RR)=2$, then the isomorphism class of $\C$ is the only root.
\item[(2)] The upper bound in Proposition~\ref{root1} is sharp. For example, let $a \ge 3$ be an integer and we consider the numerical semigroup ring $\RR = k[[\{t^i\}_{a \le i \le 2a-1}]] \subseteq k[[t]]$. Then the canonical module of $\RR$ is
\[ K_{\RR} = \sum_{i=0}^{a-2} \RR t^i = (\RR + \RR t)^{a-2}.\] Thus $\RR$ has a canonical ideal $\C= t^{a(a-2)} K_{\RR}$.  Let $L = (t^a,t^{a+1})$. Then $\C = L^{a-2}$. Hence we have   $ \tau_{L}(\C) = a-2  = r (\RR) -1$.
\end{enumerate}
}\end{Remark}

From Proposition \ref{root1} we have the following.

\begin{Theorem}\label{root}
Let $(\RR,\m)$ be a $1$-dimensional Cohen-Macaulay local ring with a canonical ideal. If $\RR$ is not Gorenstein, then
$\mathrm{root}(\RR)$ is a finite set of cardinality less than $r(\RR)$.
\end{Theorem}

\subsubsection*{Applications of roots}

\begin{Proposition}
Let $(\RR,\m)$ be a $1$-dimensional Cohen-Macaulay local ring with a canonical ideal $\C$.
Let $f$ be the supremum of the reduction numbers of the $\m$-primary ideals. Suppose that $L^{n} \simeq \C$.
If $p$ divides $n$, then $\rho(\RR) \leq (f+p-1)/p$.
\end{Proposition}

\begin{proof}
 Since $n=pm$, by replacing $L^m$ by $I$, we may assume that $I^p = \C$. Let $r=\red(I)$ with $I^{r+1} = xI^r$. Then $r=ps+q$ for some $q$ such that $-p+1 \leq q \leq 0$.
Since $ps=r-q \geq r$, we have
\[ \C^{s+1} = I^{ps+p} = x^{p} I^{ps} = x^{p} \C^{s}.\]
Thus, $ \rho(\RR)= \red(\C) \leq s = (r-q)/p \leq (r+p-1)/p \leq (f+p-1)/p$.
\end{proof}

%\begin{Conjecture}{\rm
%If $\RR$ has an even root, then $\cdeg(\RR)$ is even.
%}\end{Conjecture}

\subsubsection*{A computation of roots of the canonical ideal}

Let $0 < a_1< a_2 < \cdots < a_q$ be integers such that $\operatorname{gcd}(a_1, a_2, \ldots, a_q) = 1$. Let $H = \left<a_1, a_2, \ldots, a_q\right>$ be the numerical semigroup generated by $a_i's$.
% Hence $H = \{\sum_{i=1}^qc_ia_i \mid 0 \le c_i \in \Bbb Z~\text{for~all}~1 \le i \le q\}$.
 Let $V = k[[t]]$ be the formal power series ring over a field $k$ and set $\RR = k[[t^{a_1},t^{a_2}, \ldots, t^{a_q}]]$. We denote by $\m$ the maximal ideal of $\RR$ and by $e = a_1$ the multiplicity of $\RR$. Let $v$ be the discrete valuation of $V$.
In what follows, let $\RR \subseteq L \subseteq V$ be a finitely generated $\RR$-submodule of $V$ such that $\nu(L)>1$. We set $\ell = \nu(L)-1$. Then we have the following.

\begin{Lemma}\label{lemma1}
With notation as above, $1 \not\in \m L$.
\end{Lemma}

\begin{proof}
Choose $0 \ne g \in \m$ so that $gV \subsetneq \RR$. Then $Q = g\RR$ is a minimal reduction of the $\m$-primary ideal $I = gL$ of $\RR$, so that $g \not\in \m I$. Hence $1 \notin\m L$.
\end{proof}

\begin{Lemma}\label{2}
With notation as above, there exist elements $f_1, f_2, \ldots, f_\ell \in L$ such that
\begin{enumerate}
\item[{\rm (1)}] $L = \RR + \sum_{i=1}^\ell \RR f_i$,
\item[{\rm (2)}] $0 < v(f_1) < v(f_2) < \ldots <v(f_\ell)$, and
\item[{\rm (3)}] $v(f_i)\not\in H$ for all $1 \le i \le \ell$.
\end{enumerate}
\end{Lemma}

\begin{proof}
Let $L = \RR + \sum_{i=1}^\ell \RR f_i$ with $f_i \in L$. Let $1 \le i \le \ell$ and assume that $m= v(f_i) \in H$. We write $f_i = \sum_{j = m}^\infty c_jt^j$ with $c_j \in k$. Then $c_s \ne 0$ for some $s > m$ such that $s \not\in H$, because $f_i \not\in \RR$. Choose such integer $s$ as small as possible and set $h = f_i - \sum_{j=m}^{s-1}c_jt^j$. Then $\sum_{j=m}^{s-1}c_jt^j \in \RR$ and $\RR + \RR f_i = \RR + \RR h$. Consequently, as  $v(h) =s  > m = v(f_i)$, replacing $f_i$ with $h$, we may assume that $v(f_i) \not\in H$ for all $1 \le i \le \ell$. Let $1 \le i<j \le \ell$ and assume that $v(f_i) = v(f_j)=m$. Then, since $f_j = cf_i + h$ for some $0 \ne c \in k$ and $h \in L$ such that  $v(h) >m$, replacing $f_j$ with $h$, we may assume that $v(f_j) > v(f_i)$. Therefore we can choose a minimal system of generators of $L$ satisfying conditions (2) and (3).
\end{proof}

\begin{Lemma}\label{lemma_ad2}
%With notation as above, 
Let $f_0, f_1, f_2, \ldots, f_\ell \in V$. Assume that  $L = \left<f_0, f_1, f_2, \ldots, f_\ell\right>$ and  $v(f_0) < v(f_1) < v(f_2) < \ldots <v(f_\ell).$ Then $L = \left<1, f_1, f_2, \ldots, f_\ell\right>$.
\end{Lemma}

\begin{proof}
Since $1 \in L$ and $L \subseteq V$, we have $v(f_0) = 0$. We may assume that $f_0 = 1 + \xi$ with $\xi \in tV$. Then $\xi \in L$ as $1 \in L$. We write $\xi = \alpha_0 f_0 + \alpha_1 f_1 + \ldots + \alpha_\ell f_\ell$ with $\alpha_i \in \RR$. Then $v(\alpha_0) = v(\alpha_0 f_0) > 0$ since $f_i \in tV$ for all $1 \le i \le \ell$, so that $\alpha_0 \in \m$. Consequently $L/\m L$ is spanned by the images of $1, \{f_i\}_{1 \le i \le \ell}$, whence $L = \left<1, f_1, f_2, \ldots, f_\ell\right>$ as claimed.
\end{proof}

\begin{Proposition}\label{NS1} With notation as above, let $1, f_1, f_2, \ldots, f_\ell \in L$  and $1, g_1, g_2, \ldots, g_\ell \in L$ be systems of generators of $L$ and assume that both of them satisfy condition $(2)$ in {\rm Lemma~\ref{2}}. Suppose that $v(f_\ell) < e=a_1$. Then $v(f_i) = v(g_i)$ for all $1 \le i \le \ell$.
\end{Proposition}

\begin{proof} We set $m_i = v(f_i)$ and $n_i = v(g_i)$ for each $1 \le i \le \ell$.
Let us write
\begin{eqnarray*}
f_1&=& \alpha_0 + \alpha_1 g_1 + \ldots + \alpha_\ell g_\ell\\
g_1&=& \beta_0 + \beta_1 f_1 + \ldots + \beta_\ell f_\ell
\end{eqnarray*}
with $\alpha_i, \beta_i \in \RR$. Then $v(\beta_1 f_1 + \ldots + \beta_\ell f_\ell) \ge v(f_1) = m_1>0$, whence $\beta_0 \in \m$ because $v(g_1)=n_1 > 0$. We similarly have that $\alpha_0 \in \m.$ Therefore $n_1 = v(g_1) \ge m_1$, since $v(\beta_0) \ge e > m_\ell \ge m_1$ and $v(\beta_1 f_1 + \ldots + \beta_\ell f_\ell) \ge m_1$. Suppose that $n_1 > m_1$. Then  $v(\alpha_1 g_1 + \ldots + \alpha_\ell g_\ell) \ge n_1 > m_1$ and $v(\alpha_0) \ge e > m_1$, whence $v(f_1)>m_1$, a contradiction. Thus $m_1 = n_1$.
%$m_1= v(f_1) = v(\alpha_0 + \alpha_1 g_1 + \ldots + \alpha_\ell g_\ell) > m_1$. This is impossible.

Now let $1 \le i < \ell$ and assume that $m_j = n_j $ for all $1 \le j \le i$. We want to show $m_{i+1} = n_{i+1}$. Let us write
\begin{eqnarray*}
f_{i+1}&=& \gamma_0 + \gamma_1 g_1 + \ldots + \gamma_\ell g_\ell\\
g_{i+1}&=& \delta_0 + \delta_1 f_1 + \ldots + \delta_\ell f_\ell
\end{eqnarray*}
with $\gamma_i, \delta_i \in \RR$.

First we claim that $\gamma_j, \delta_j \in \m$ for all $0 \le j \le i$.
%\begin{Claim}\label{1.4}$\gamma_j, \delta_j \in \m$ for all $0 \le j \le i$.
%\end{Claim}
%\begin{proof}
 As above, we get $\gamma_0, \delta_0 \in \m$.  Let $0 \le k < i$ and assume that $\gamma_j, \delta_j \in \m$ for all $0 \le j \le k$. We show $\gamma_{k+1}, \delta_{k+1} \in \m$. Suppose that $\delta_{k+1} \not\in \m$. Then as $v(\delta_0+ \delta_1 f_1 + \ldots + \delta_k f_k) \ge e > m_{k+1}$,  we get  $v(\delta_0 + \delta_1 f_1 + \ldots + \delta_k f_k + \delta_{k+1} f_{k+1})=m_{k+1}$, so that $n_{i+1} = v(g_{i+1})=v(\delta_0+ \delta_1 f_1 + \ldots + \delta_\ell f_\ell) = m_{k+1}$, since $v(f_h) =m_h> m_{k+1}$ if $h> k+1$. This is impossible, since $n_{i+1} > n_i = m_i \ge m_{k+1}$. Thus $\delta_{k+1} \in \m$. We similarly get $\gamma_{k+1} \in \m$, and the claim is proved.
 % (In fact, assume that $\gamma_{k+1} \not\in \m$. Then $v(\gamma_0 g_0 + \gamma_1 g_1 + \ldots + \gamma_k g_k) \ge e > m_{k+1}=n_{k+1}$,  we get  $v(\gamma_0 g_0 + \gamma_1 g_1 + \ldots + \gamma_k g_k + \gamma_{k+1} g_{k+1})=n_{k+1}$, so that $m_{i+1} = v(f_{i+1}) = v(\gamma_0 g_0 + \gamma_1 g_1 + \ldots + \gamma_\ell g_\ell) = n_{k+1}$, because $v(g_j) =n_j > n_{k+1}$ if $j > k+1$. This is impossible, since $m_{i+1} > m_i = n_i \ge n_{k+1}$.)
 %Thus Claim \ref{1.4} follows.
%\end{proof}

Consequently, since $v(\delta_0 + \delta_1 f_1 + \ldots + \delta_i f_i) \ge e > m_{i+1}$ and $v(f_h) \ge  m_{i+1}$ if $h\ge i+1$, we have $n_{i+1} = v(g_{i+1}) = v(\delta_0 + \delta_1 f_1 + \ldots + \delta_\ell f_\ell)\ge m_{i+1}$. Assume that $n_{i+1} > m_{i+1}$. Then since $v(\gamma_0 + \gamma_1 g_1 + \ldots + \gamma_i g_i) \ge e > m_{i+1}$ and $v(g_h) \ge  n_{i+1} > m_{i+1}$ if $h\ge i+1$, we have $m_{i+1} = v(f_{i+1}) = v(\gamma_0 + \gamma_1 g_1 + \ldots + \gamma_\ell g_\ell) > m_{i+1}$. This is a contradiction. Hence $m_{i+1} = n_{i+1}$, as desired.\end{proof}

 \begin{Theorem}\label{th1-1}
With notation as above, assume that $r (\RR) = 3$ and write the canonical module $K_{\RR} = \left<1, t^a,t^b\right>$ with $0 < a < b$. Suppose that $b \ne 2a$ and $b < e$. Let $J$ be an ideal of $\RR$ and let $n \geq 1$ be an integer. If $J^n \cong K_{\RR}$, then $n = 1$.
\end{Theorem}

\begin{proof}
We already know $n \leq 2$ by Proposition~\ref{root1}. Suppose that $n=2$. Let $f \in J$ such that $JV = fV$ and set $L= f^{-1}J$. Then $\RR \subseteq L \subseteq V$ and $L^2 = f^{-2}J^{2} \cong K_{\RR}$. By Lemma \ref{2}, we can write $L = \left<1, f_1, f_2, \ldots, f_{\ell} \right>$,  where $\ell = \nu(L) -1 \ge 1$, $0 < v(f_1) < v(f_2) < \cdots < v(f_{\ell} )$, and $v(f_i) \not\in H$ for all $1 \leq i \leq \ell$. Then $L^2 = \left<1, \{f_i\}_{1 \leq i \leq \ell}, \{f_if_j\}_{1 \leq i \leq j \leq \ell} \right>$. Let $n_i = v(f_i)$ for each $1 \leq i \leq \ell$.

\begin{claim}\label{claim1} Suppose that $\ell > 1$. Then $f_1 \not\in \left<1, \{f_i\}_{2 \leq i \leq \ell}, \{f_if_j\}_{1 \leq i \leq j \leq \ell} \right>$.
\end{claim}

\begin{proof}[Proof of Claim \ref{claim1}]
Assume the contrary and write $$f_1 = \alpha + \sum_{i=2}^{\ell}\alpha_i f_i + \sum_{1 \leq i \leq j \leq \ell}\alpha_{ij}f_if_j$$ with $\alpha, \alpha_i, \alpha_{ij} \in \RR.$ Then since $n_1 < n_i$ for $i \geq 2$ and $n_1 \leq n_i < n_i+n_j$ for $1 \leq i \leq j \leq \ell$, we have $n_1 = v(\alpha)$, which is impossible, because $n_1 \not\in H$ but $v(\alpha) \in H$.
\end{proof}

\noindent Since $\nu(L^2) = 3$, we have $L^2=\left<1, f_1, f_if_j\right>$ for some $1 \leq i \leq j \leq \ell$. This is clear if $\ell=1$, and it follows by Claim \ref{claim1} if $\ell>1$. In fact, the other possibility is $L^2 = \left<1,f_1,f_i\right>$ with $i \geq 2$. When this is the case, we get $L = L^2$, and so $\operatorname{red}_{(f)}J \leq 1$, whence $\operatorname{red}_{(f^2)}J^2 \leq 1$. Therefore, since $J^2 \cong K_{\RR}$, $\RR$ is a Gorenstein ring, which is a contradiction.

We now choose $0 \neq \theta \in \mathrm{Q}(V)$, where $\mathrm{Q}(V)$ is the quotient field of $V$, so that $K_{\RR} = \theta L^2$. Notice that $\theta$ is a unit of $V$ (as $\RR \subseteq L$ and $\RR \subseteq K_{\RR}$). Compare
\[ K_{\RR} = \left<1, t^a, t^b\right> = \left<\theta, \theta f_1, \theta f_if_j\right>\]
and notice that $\left<\theta, \theta f_1, \theta f_if_j\right>= \left<1, \theta f_1, \theta f_if_j\right>$ by Lemma \ref{lemma_ad2}. Therefore since $0 < a < b < e$ and $0< n_1 < n_i + n_j$, by Proposition \ref{NS1} $n_1 = a$ and $n_i + n_j = b$; hence $b \ge 2a$. Furthermore, since $b \ne 2a$, we have $\ell>1$.

Since $L \subseteq L^2$, we have $f_2 \in L^2 =\left<1, f_1, f_if_j\right>$. Let us write
$$f_2 = \alpha + \beta f_1 + \gamma f_if_j$$
with $\alpha, \beta, \gamma \in \RR$. Then $\alpha \in \m$, because $v(f_2) = n_2 > n_1$ and $v(\beta f_1 + \gamma f_if_j) \ge  n_1$. Hence $v(\alpha) > b$, because $e > b$. Consequently $\beta \in \m$; otherwise $n_2= v(f_2) = v(\beta f_1) = n_1$ (notice that $v(\alpha + \gamma f_if_j) \ge b \ge 2n_1$). Therefore $v(\beta f_1) >b+  n_1$, so that $n_2 = v(f_2)\ge b$. Since $b = n_i + n_j > n_j$, we then have $i=j=1$. This is contradiction, since $b \ne 2a$.
\end{proof}

Let us give examples satisfying the conditions stated in Theorem \ref{th1-1}.

\begin{Example}{\rm Let $e \geq 7$ be an integer and set
\[ H = \left<e + i \mid 0 \le i \le e-2 ~\text{such~that}~ i \ne e-4, e-3 \right>.\]
 Then $K_{\RR} = \left<1, t^2, t^3 \right>$ and $r (\RR) =3$. More generally, let $a,b, e \in \Bbb Z$ such that $0 < a< b$, $b < 2a$, and $e \geq a+b+2$. We consider the numerical semigroup
 \[ H = \left< e+i \mid 0 \leq i \leq e-2~\text{such~that}~ i \neq e-b-1, e-a-1\right>.\] Then $K_{\RR} =\left <1, t^a, t^b \right>$ and $r (\RR) = 3$. These rings $\RR$ contain no ideals $L$ such that $L^n \cong K_{\RR}$ for some integer $n \geq 2$.
}\end{Example}

%%--------------------
%\begin{Remark}{\rm
%Which semi-dualizing ideals are roots? (\cite{CW})
%}\end{Remark}
%%----------------------

\section{Change of rings}
\noindent
Degree formulas are often  statements about change of rings.
Let us see how this may work for $\cdeg(\RR)$. Let $(\RR, \m)$ be a Cohen-Macaulay local ring of dimension $d \geq 1$ with a canonical ideal $\C$. Let $(\AA, \n)$ be a Cohen-Macaulay local ring with a finite injective morphism $\varphi: \RR \rar \AA$. Suppose that the total ring of fractions of $\AA$
is Gorenstein. Then a canonical module of $\AA$ is $\mathcal{D} = \Hom_{\RR}(\AA, \C)$, which is isomorphic to an ideal of $\AA$.

\begin{Example}{\rm  Let $\AA = K[x_1, \ldots, x_d]^{(n)}$ be the $n$-Veronese subring of the
polynomial ring in $d$ variables. Fix the Noether normalization $\RR = K[x_1^n, \ldots, x_d^n]$. The canonical module of $\AA$ is $\Hom_{\RR}(\AA, \RR)$.
}\end{Example}

We want to develop relationships between the pairs
\{$\cdeg(\RR)$, $\red(\C)$\} and \{$\cdeg(\AA)$, $\red(\mathcal{D})$\}.

\subsubsection*{Polynomial extension} Let us compare $\cdeg(\RR)$ to $\cdeg(\RR[x]_{\m \RR[x]})$ and $\cdeg(\RR[[x]])$.

\begin{Lemma}
Let $(\RR, \m)$ be a Cohen-Macaulay local ring of dimension $d \geq 1$ with a canonical ideal $\C$.
Let $\varphi: \RR \rar \AA$ be a flat injective morphism, where $\AA=\RR[x]_{\m \RR[x]}$ or $\AA=\RR[[x]]$.

\begin{enumerate}
\item[{\rm (1)}] The morphism $\varphi$ has Gorenstein fiber.
\item[{\rm (2)}] $\C\otimes \AA$ is a canonical ideal of $\AA$.
\item[{\rm (3)}] If $\C$ is equimultiple, then for the minimal reduction $(a)$ of $\C$,
\[ \deg(\C/(a)) = \deg(\C/(a) \otimes \AA).\]
\end{enumerate}
\end{Lemma}

\begin{proof}
Recall that if
$\AA = \RR[x]_{\m \RR[x]}$, then $\AA/\m \AA= k(x)$, and if $\AA = \RR[[x]]$,  then $\AA/\m \AA = k [[x]]$.
This proves (1). Notice that the embedding $\C \subset \RR$ leads to the embedding $\C\otimes \AA\subset \AA$. Assertion (2) follows from
\cite[Theorem 4.1]{Aoyama}, or \cite[Satz 6.14]{HK2}.
\end{proof}

\begin{Proposition}
Let $(\RR, \m)$ be a Cohen-Macaulay local ring of dimension $d \geq 1$ with a canonical ideal $\C$.
Then $\cdeg(\RR)= \cdeg(\RR[x]_{\m \RR[x]})$.
\end{Proposition}

\begin{proof}
 Since $\Min(\C \otimes \AA ) = \{ \p \AA\mid \p \in \Min(\C)\} $ and
$\deg(\RR/\p) = \deg(\AA/\p \AA)$,  $\cdeg(\RR)$ and $\cdeg(\AA)$ are defined by the same summations.
\end{proof}

\begin{Question}{\rm
Is it true that $\cdeg(\RR)= \cdeg(\RR[[x]])$? If $\C$ is equimultiple, then it is true. What if $\C$ is not equimultiple?
}\end{Question}

\subsubsection*{Completion}

\begin{Proposition}\label{th1}
Let $(\RR,\m)$ be a $1$-dimensional Cohen-Macaulay local ring with infinite residue field and a canonical ideal. Then $\cdeg (\widehat{\RR}) = \cdeg (\RR)$.
\end{Proposition}

\begin{proof} Let $(a)$ be a minimal reduction of  the canonical ideal $\C$. Then $a \widehat{\RR}$ is a minimal reduction of the canonical ideal $\C \widehat{\RR}$ of $\widehat{\RR}$. Thus, we have
\[ \cdeg (\widehat{\RR}) = \l_{\widehat{\RR}} ( \C \widehat{\RR} / a \widehat{\RR} ) = \l_{\widehat{\RR}}\left( \widehat{\RR} \otimes_{R} \C/(a) \right) = \l_{\RR}(\C/(a)) \l_{\widehat{\RR}} \left( \widehat{\RR}/\m \widehat{\RR} \right) = \cdeg(\RR).\]

\end{proof}

\begin{Theorem}\label{th1}
Let $(\RR,\m)$ be a Cohen-Macaulay local ring of dimension $d \geq 1$ with infinite residue field and a canonical ideal. Then we have the following.
\begin{enumerate}
\item[{\rm (1)}]$\cdeg (\widehat{\RR}) \geq \cdeg (\RR)$.
\item[{\rm (2)}]$\cdeg (\widehat{\RR}) = \cdeg (\RR)$ if and only if  for every $\p \in \Spec(\RR)$ with $\dim \RR_\p=1$, $\RR_\p$ is a Gorenstein ring, or $\p \widehat{\RR} \in \Spec (\widehat{\RR})$.
\end{enumerate}
\end{Theorem}

\begin{proof} Let $\C$ be a canonical ideal of $\RR$. We set $X = \Ass_{\RR} \RR/\C$ and $Y = \Ass_{\widehat{\RR}}\widehat{\RR}/ \C\widehat{\RR}$. Then
\[ Y = \bigcup_{\p \in X}\Ass_{\widehat{\RR}}\widehat{\RR}/\p \widehat{\RR}.\]
Notice that for each $\p \in X$, $Z=\Ass_{\widehat{\RR}}\widehat{\RR}/\p \widehat{\RR} = \{P \in Y \mid P \cap \RR = \p\}$.

\noindent We now fix an element $\p \in X$ and let $S = \widehat{\RR} \setminus \bigcup_{P \in Z} P$. We set $A = \RR_\p$ and $B = S^{-1}\widehat{\RR}$. Then we have a flat homomorphism
\[ A=\RR_\p \to \widehat{\RR}_\p \to B=S^{-1}\widehat{\RR},\] since the homomorphism $\widehat{\RR}_\p \to B$ is a localization (by the image of $S$ under the map $\widehat{\RR} \to \widehat{\RR}_\p$).
 We choose an element $a \in \C$ so that $a \RR_\p$ is a reduction of $\C \RR_\p$. We then have that
 \[ \l_B(B \otimes_AX) = \l_A(X){\cdot} \l_B(B/\p B) \]
  for every $A$-module $X$ with $\l_A(X) < \infty$ and that
  \[ \l_B(B/\p B) = \sum_{P \in Z}   \l_{\widehat{\RR}_P}(\widehat{\RR}_P/\p \widehat{\RR}_P).\] Hence
\[ \sum_{P\in Z}\l_{\widehat{\RR}_P}(\C\widehat{\RR}_P/a\widehat{\RR}_P) = \l_B(\C B/aB) = \l_B(B/\p B) {\cdot}\l_A(\C A/aA).\]
Consequently
\[\begin{array}{rcl}
{\ds  \cdeg ( \widehat{\RR}) = \sum_{\p \in X}\left(\sum_{P \in Z}\l_{\widehat{\RR}_P}(\C \widehat{\RR}_P/a\widehat{\RR}_P) \right)}  &=& {\ds  \sum_{\p \in X}\l_B(B/\p B){\cdot}\l_A(\C A/aA) } \\ && \\
&\geq& {\ds  \sum_{\p \in  X}\l_{\RR_\p}(\C\RR_\p/a \RR_\p) = \cdeg(\RR),}
\end{array}\]
%%-----------------------
%\begin{eqnarray*}
%\cdeg ( \widehat{\RR}) &=& \sum_{\p \in X}\left(\sum_{P \in Z}\l_{\widehat{\RR}_P}(\C \widehat{\RR}_P/a\widehat{\RR}_P) \right)\\
%&=& \sum_{\p \in X}\l_B(B/\p B){\cdot}\l_A(\C A/aA)\\
%&\geq& \sum_{\p \in  X}\l_{\RR_\p}(\C\RR_\p/a \RR_\p)\\
%&=& \cdeg(\RR),
%\end{eqnarray*}
%%---------------------------
which shows assertion (1).

\medskip

\noindent We have
\begin{eqnarray*}
\cdeg(\widehat{\RR}) = \cdeg(\RR) &\Leftrightarrow&\forall \p \in  X=\Ass_{\RR}(\RR/\C), \RR_\p~\text{is~a~Gorenstein~ring,~or}~\l_B(B/\p B) = 1\\ &\Leftrightarrow&\forall \p \in X, \RR_\p ~\text{is~a~Gorenstein~ring,~or}~\Ass_{\widehat{\RR}}\widehat{\RR}/\p \widehat{\RR}~\\ &{}&~\text{is~a~singleton~and}~\widehat{\RR}_P/\p \widehat{\RR}_P~\text{is~a~field}~\text{where}~\{P\} = \Ass_{\widehat{\RR}}\widehat{\RR}/\p \widehat{\RR}\\
&\Leftrightarrow&\forall \p \in X, \RR_\p~\text{is~a~Gorenstein~ring,~or}~\p \widehat{\RR} \in \Spec (\widehat{\RR}).
\end{eqnarray*}

\noindent Assume $\cdeg(\widehat{\RR})= \cdeg(\RR)$. Let $\p \in \Spec (\RR)$ with $\dim \RR_\p = 1$ and take a non-zerodivisor $f \in \p$. Then  since $f \C  \cong \C$ and $\p \in \Ass_{\RR} \RR/f \C$, $\RR_\p$ is a Gorenstein ring or $\p \widehat{\RR} \in \Spec( \widehat{\RR})$.
 Thus assertion (2) follows.
\end{proof}

\subsubsection*{Augmented ring}%Idealization of $\m$ over $\RR$

Motivated by \cite[Theorem 6.5]{GMP11} let us determine the canonical degree of a class of local rings.

Let $\RR$ be a commutative ring with total quotient ring $\rmQ(\RR)$ and let $\mathcal F$ denote the set of $\RR$-submodules of $\rmQ(\RR)$. Let $M,K \in \mathcal F$. Let $M^\vee = \Hom_{\RR}(M, K)$ and let $\AA = \RR \ltimes M$ denote the idealization of $M$ over $\RR$. Then the $\RR$-module $M^\vee \oplus K$  becomes an $\AA$-module under the action $$(a,m)\circ (f,x) = (af, f(m)+ax),$$ where $(a,m) \in \AA$ and $(f,x) \in M^\vee \times K$. We notice that the canonical homomorphism
$\varphi : \Hom_{\RR}(\AA,K) \to  M^\vee \times K$
such that $\varphi(f) = (f\circ \lambda, f(1))$ is an $\AA$-isomorphism, where $\lambda : M \to \AA, \lambda(m) = (0,m)$. We also notice that $K:M \in \mathcal F$ and $(K:M) \times K \subseteq \rmQ(\RR) \ltimes \rmQ(\RR)$ is an $\AA$-submodule of $\rmQ(\RR) \ltimes \rmQ(\RR)$, the idealization of $\rmQ(\RR)$ over itself. When $\rmQ(\RR){\cdot}M = \rmQ(\RR)$, identifying $\Hom_{\RR}(M,K)$ with $K:M$, we have a natural isomorphism of $\AA$-modules
$$\Hom_{\RR}(\AA,K) \cong (K:M) \times K.$$

\begin{Proposition}\label{canonical-idealization}
Let $(\RR,\m)$ be a Cohen-Macaulay local ring possessing the canonical module $\rmK_{\RR}$. Let $M,K$ be $\RR$-submodules of $\rmQ(\RR)$.
Assume that $M$ is a finitely generated $\RR$-module with $\rmQ(\RR){\cdot}M = \rmQ(\RR)$ and that $K \cong \rmK_{\RR}$ as an $\RR$-module. Let $\AA = \RR \ltimes M$, and let $L = K:M$. Then $\rmK_{\AA}= (K:M) \times K$ in $\rmQ(\RR) \ltimes \rmQ(\RR)$ is a canonical module of $\AA$, and $(\rmK_{\AA})^n = L^n \times L^{n-1}K$ for all $n \ge 1$.
\end{Proposition}

\begin{proof} The assertions follow from the above observations. The proof of the equality $(\rmK_{\AA})^n = L^n \times L^{n-1}K$ follows by induction on $n$. \end{proof}

\begin{Theorem}\label{ch1}
Let $(\RR,\m)$ be a $1$-dimensional Cohen-Macaulay local ring with infinite residue field and a canonical ideal $\C$.
Suppose $\RR$ is not a $\mathrm{DVR}$. Then we have the following:
%Let $\AA$ be the augmented ring $\RR \ltimes \m$. Then the canonical degrees and Cohen-Macaulay types of $\RR$ and  $\AA$ are related in the following manner:
\[ \cdeg(\RR \ltimes \m)= 2 \, \cdeg(\RR)+2 \quad \mbox{\rm and} \quad r(\RR \ltimes \m) = 2 \, r(\RR)+1.\]
\end{Theorem}

\begin{proof} We may assume that $\C\subset \m^2$ by replacing $\C$ with $b\C$ if necessary, where $b \in \m^2$ is a non-zerodivisor of $\RR$.  Let $\AA=\RR \ltimes \m$ and set $L=\C:\m$.
Then by Proposition \ref{canonical-idealization}
$\mathcal{D} = L \times  \C$ is a canonical ideal of $\AA$ (notice that $L \subset \m$, since $\C \subset \m^2$).  Let $a \in \C$ and assume that $(a)$ is a reduction of $\C$. Then since $L^2 = \C L$ by \cite[Lemma 3.6 (a)]{CHV} (notice that $\nu(\m/\C) \ge 2$, since $\RR$ is not a DVR), the ideal $(a)$ is also a reduction of $L$, so that $a\AA$ is a minimal reduction of the canonical ideal $\mathcal{D}$. Hence
\[ \cdeg(\AA) = \lambda(\mathcal{D}/a\AA) =\lambda(L/a\RR) + \lambda(\C/a\m).\]
Because  $\lambda(L/\C)=\lambda(\RR/\m)=1$, we get
\[\lambda(L/a\RR) + \lambda(\C/a\m) = [\lambda(L/\C) + \lambda(\C/a\RR)] + [\lambda (\C/a\RR) +\lambda(a\RR/a\m)] =  2\lambda(\C/a\RR) + 2.\]
Thus ${\ds \cdeg(\AA)  = 2 \, \cdeg(\RR) + 2}$. Notice that $\n=\m \times \m$ is the maximal of $\AA$. Since $\lambda(L/\C)=1$ and $\m \C = \m L$ by \cite[Lemma 3.6 (b)]{CHV}, we have
\[ \begin{array}{rcl}
{\ds r (\AA) = \l(\mathcal{D}/\n \mathcal{D} ) = \l((\C/\m \C) \oplus (L/\m L))} &=& {\ds \l(\C/\m \C) + \l(L/\m \C)} \\
 &=& {\ds \l(\C/\m \C) + [\l(L/\C) + \l(\C/ \m \C)] = 2 \, r(\RR) + 1}
 \end{array}\]
as claimed.
\end{proof}

\begin{Corollary}{\rm (\cite[Theorem 6.5]{GMP11})} Let $(\RR,\m)$ be a $1$-dimensional Cohen-Macaulay local ring with a canonical ideal. Then
$\RR$ is an almost Gorenstein ring if and only if $\RR \ltimes \m$ is an almost Gorenstein ring.
\end{Corollary}

\begin{proof} It follows from Theorem~\ref{ch1} and Proposition~\ref{1dimalmostg}.
\end{proof}

%\begin{Remark}{\rm Let $(\RR, \m)$ be a $1$-dimensional Cohen-Macaulay local ring with a canonical ideal $\C$. Let $\AA=\RR \ltimes \m$ and $L= \C : \m$. Then, as proved in Theorem~\ref{ch1}, $\mathcal{D}= L \times \C$ is a canonical ideal of $\AA$. Suppose that $H=(h, a)$ is a minimal reduction of $\mathcal{D}$ with reduction number $n$.

%{\bf REVISE} Then, by Proposition \ref{canonical-idealization}, \[ \mathcal{D}^{n+1}= L^{n+1} \times L ^n\C
% =H \mathcal{D}^n = a\C^n \ltimes (h \C^n +a L \C^{n-1}), \]  which shows that $(a)$ is a minimal reduction of $\C$ and $\red(\C) \leq \red(\mathcal{D})$.

 %On the other hand, assume that $(a)$ is a minimal reduction of $\C$ with reduction number $n-1$. Then $(a,0)$ is a minimal reduction of $\mathcal{D}$ with $\mathcal{D}^{n+1}=H \mathcal{D}^n$. Thus $\red(\C) \leq \red(\mathcal{D}) \leq \red(\C)+1$.}
%\end{Remark}

%\begin{Example}{\rm Let $\RR$ be as in Example \ref{Ex1linkage}. Then $\C=(x, z)$ is a canonical ideal of $\RR$ with a minimal reduction $(x)$ and $\rho(\RR)=2$. We have that $C : \m=(x,y,z)=\m$, $\mathcal{D}^{3}= \C^{3} \ltimes \m \C^2$ and $(x,0)\mathcal{D}^{2}=x\C^2\ltimes x\m\C$. It is easy to check that $\m\C^2=x\m\C$, so that $\rho(\RR \ltimes \m)=2$. }
%\end{Example}

\begin{Question}{\rm
Let $(\RR, \m)$ be a $1$-dimensional Cohen-Macaulay local ring with a canonical ideal.
(i) Is it true that $\rho(\RR)= \rho(\RR \ltimes \m)$?  (ii) Are the roots of $\RR \ltimes \m$ related to the roots of $\RR$?
}\end{Question}

%\begin{proof}
%Let $\C$ be the canonical ideal of $\RR$. Let $\AA=\RR \ltimes \m$ and $L= \C : \m$. Then, as proved in Theorem~\ref{ch1}, $\mathcal{D}= \C \ltimes L $ is a canonical ideal of $\AA$. It is enough to show that $\red(\C)=\red(\mathcal{D})$.
 %If $H=(h, a)$ is a minimal reduction of $\mathcal{D} = \C \ltimes L$, say
 %\[ \mathcal{D}^{n+1}= \C^{n+1} \ltimes L \C^n
 % =H \mathcal{D}^n = h \C^n \ltimes a \C^n +h L \C^{n-1}, \] and $(h)$ must be a minimal reduction of $\C$.
  %Note that $H = (h,0)$ will work out.
 % If $n= \red(\C)+1$, $hL \C^{n-1} =L \C^n$. Thus $\red(\C) \leq \red(\mathcal{D}) \leq \red(\C)+1$.
%\end{proof}

%\begin{Question}{\rm (i) Are the roots of $\RR \ltimes \m$ related to the roots of $\RR$?}
%\begin{enumerate}
%\item[(1)] Consider $\AA = \RR\ltimes \m^{\oplus s}$. Note that the canonical ideal of $\AA$ is
% $\C \oplus L^{\oplus s}$.  Repeat the calculation to get $\cdeg(\AA )$ and $r(\AA)$.
%\item[(2)] Are the roots of $\RR \ltimes \m$ related to the roots of $\RR$?
%\end{enumerate}
%\end{Question}

%{\bf From Shiro:}
%I feel I have an affirmative answer to Question 6.11(i) for the case where $R$ is a numerical semigroup ring. Will homogenize notation.
%I put this at the end of the file

\subsubsection*{Hyperplane sections}
 A change of rings issue is the  comparison $\cdeg(\RR)$ to $\cdeg(\RR/(x))$ for an
appropriate regular element $x$. We know that if $\C$ is a canonical module for $\RR$ then
$\C/x \C$ is a canonical module for $\RR/(x)$ with the same number of generators, so type is preserved under
specialization. However $\C/x \C$ may not be isomorphic to an ideal of $\RR/(x)$. Here is a case of good behavior. Suppose
$x$ is regular modulo $\C$. Then for the sequence
\[ 0 \rar \C \lar \RR \lar \RR/\C \rar 0,\]
we get the exact sequence
\[ 0 \rar \C/x\C  \lar \RR/(x) \lar \RR/(\C,x) \rar 0,\] so the canonical module
$\C/x\C$ embeds in $\RR/(x)$.
 Note that this leads $\red(\C) \geq \red(\C/x\C)$.

 \medskip

\begin{Proposition}\label{hscdeg} Suppose $(\RR, \m)$ is a Cohen-Macaulay local ring of dimension $d \geq 2$ that has a canonical ideal $\C$. Suppose that $\C$ is equimultiple and  $x$ is regular modulo $\C$. Then
\[ \cdeg(\RR) \leq \cdeg(\RR/(x)).\]
\end{Proposition}

\begin{proof}
Let $(a)$
be a minimal reduction of $\C$. Then its image in $\RR/(x)$ is a minimal reduction of $\C/x\C$. Since $x$ does
not belong to any minimal primes of $\C$, which are the same as the minimal primes of $(a)$,
the sequence $a,x$ is a regular sequence of $\RR$. Therefore, so is $x, a$. Thus
by
 \cite[Theorem 3.2]{FW93}
\[
\cdeg(\RR/(x))= \deg(\C/(x\C)\otimes \RR/(a)) = \deg(\C/(a,x)\C) =
\deg(\C/(a) \otimes \RR/(x))\geq \cdeg(\RR).\]
\end{proof}

\begin{Question}\label{Bertinicdeg}{\rm
 Another question is when there exists $x$ such that $\cdeg(\RR) = \cdeg(\RR/(x))$ in $2$ cases: (i) $\C$ is equimultiple and (ii) $\C$ is not necessarily equimultiple. It would be a Bertini type theorem.
 }\end{Question}

\section{Relative canonical degrees: Extensions/Variations }

\noindent It is clear that the definition of canonical degree (see Theorem~\ref{gencdeg1}) does not take into account deeper properties of
$\RR$. We have seen this in the case of normal rings when  $\cdeg(\RR) = 0$ is all we get, no additional information about
$\RR$ is forthcoming--except if $\C$ is equimultiple.
At a minimum, we would like to say that if $\RR$ is Cohen-Macaulay then $\cdeg(\RR) = 0$ if and only if
$\RR$ is Gorenstein. Here are some  proposals, beginning with a generalization of  Proposition~\ref{cdeg}.

%\subsubsection*{Relative canonical degrees(Case 1)}

\begin{Definition}\label{cdeg2}{\rm Let $(\RR, \m)$ be a Cohen-Macaulay local ring of dimension $d \geq 1$ that admits a canonical ideal $\C$.
 Let $G = \gr_{\C}(\RR) = \bigoplus_{n\geq 0} \C^n/\C^{n+1} $ be the associated graded ring of $\C$, and
$M_{\C} = (\m, G_{+})$ its maximal irrelevant ideal. The {\em canonical degree}$^*$ of $\RR$ is the integer
\[ \cdeg^{*}_{\C}(\RR) := \e_0(M_{\C}, \gr_{\C}(\RR)) - \deg(\RR/\C) =\deg(\gr_{\C}(\RR)) - \deg(\RR/\C)
.\]
}\end{Definition}

Before we discuss cases, let us recall some elementary facts about the calculation of multiplicities.

\begin{Proposition}\label{degIQ}
Let $(\RR,\m)$ be a Cohen-Macaulay local ring of dimension $d \geq 1$ and let $I$ be an ideal of positive
codimension. Let $Q$ be  a reduction of $I$.

\begin{enumerate}
\item[{\rm (1)}] If $a$ is a regular element, then ${\ds  \deg(I/aI) = \deg(\RR/(a))}$.

\item[{\rm (2)}] $\deg(\gr_I(\RR)) = \deg(\gr_Q(\RR))$.

\item[{\rm (3)}] If $I$ is equimultiple and $Q$ is generated by a regular sequence, then
$\deg(\gr_I(\RR)) = \deg(\RR/Q).$
\end{enumerate}
\end{Proposition}

\begin{proof}
(1) %We give a proof for the case when the codimension of $I$ is $1$, but it works in general.
Consider the two exact sequences of modules % of dimension $d-1$
\[ 0\rar I/aI \lar \RR/aI \lar \RR/I \lar 0,\]
\[ 0\rar (a)/aI=\RR/I \lar \RR/aI \lar \RR/(a) \lar 0.\]
These yield
\[ \deg(\RR/aI) = \deg(I/aI) + \deg(\RR/I) = \deg(\RR/I) + \deg(\RR/(a)).\]
(2) and (3)  Consider the embedding
\[ 0 \rar
\AA=\RR[Qt, t^{-1}] \lar
\BB= \RR[It, t^{-1}] \lar
L \rar 0.\]
Then $L$ is a module over $\AA$ of dimension $\leq d$. We consider the snake diagram defined by
multiplication by $t^{-1}$:
\[ 0 \rar L_1
\lar \AA/t^{-1}\AA = \gr_ Q(\RR)
\lar \BB/t^{-1}\BB = \gr_ I(\RR) \lar L_2 \rar 0
\]
where $L_1$ and $L_2$ arise from
\[ 0 \rar L_1 \lar L \stackrel{t^{-1}}{\lar} L \lar L_2\rar 0,\]
which is a nilpotent action. Note that $\dim L = \dim L_1 = \dim L_2$,
and therefore $\deg(L_1) = \deg(L_2)$.
The calculation of multiplicities then gives $\deg(\gr_Q(\RR)) = \deg(\gr_I(\RR))$ in general, and
$\deg(\gr_Q(\RR)) = \deg(\RR/Q)$ in the equimultiple case.
\end{proof}

%There  is a better way to explain the equality of $\cdeg$ and $\cdeg^*$ in case $\C$ is equimultiple. It uses the following fact:

\begin{Corollary}
Let $(\RR, \m)$ be a Cohen-Macaulay local ring of dimension $d \geq 1$ that admits a canonical ideal $\C$. If $\C$ is equimultiple, then $\cdeg^{*}_{\C}(\RR)  = \cdeg(\RR)$. In particular, $\cdeg^{*}_{\C}(\RR)$ is independent of the choice of $\C$.
\end{Corollary}

\begin{proof} Let $(a)$ be a minimal reduction of $\C$. Then by Proposition~\ref{degIQ} and Theorem~\ref{gencdeg1}, we obtain
\[ \cdeg^{*}_{\C}(\RR)  = \deg(\gr_{\C}(\RR)) - \deg(\RR/\C) = \deg(\RR/(a)) - \deg(\RR/\C) = \cdeg(\RR).\]
\end{proof}

\begin{Example}{\rm
Define $\varphi: S=k[t_1, \cdots, t_5] \rar A=k[x,y]$ as $\varphi(t_i)=x^{5-i}y^{i-1}$ for each $i$. Let $L=\ker(\varphi)$ and $\RR=\mbox{Im}(\varphi)$. Then a canonical ideal $\C$ of $\RR$ is
${\ds \C = (x^{3}y^{5},\; x^{4}y^{4},\; x^{5}y^{3} ) }$
which has a minimal reduction ${\ds Q = (x^{5}y^{3}+ x^{3}y^{5},\; x^{4}y^{4} )}$.
Using Macaulay 2, we obtain
\[ \deg(\gr_{Q}(\RR)) =  6 = \deg(\RR/\C)\]
Hence  we have
\[ \cdeg^{*}_{\C}(\RR) = \deg(\gr_{Q}(\RR)) - \deg(\RR/\C) = 0.\]
Note that $\RR$ is a non-Gorenstein ring of type $3$.
}\end{Example}

%\subsubsection*{Relative canonical degrees(Case 2)}

Another family of relative $\cdeg$s  arises from the following construction.

\begin{Remark}{\rm
Let $(\RR, \m)$ be a Cohen-Macaulay local ring of dimension $d \geq 2$ that admits a canonical ideal $\C$.
Let $\mathbf{x} = \{x_1, \ldots, x_{d-1}\}\subset \m \setminus \m^2$
 be a minimal system of parameters for $\RR/\C$. Consider
$\cdeg(\RR/(\mathbf{x}))$.
 Now define
\[ \cdeg_{\C}^{\sharp}(\RR) = \min\{\cdeg(\RR/(\mathbf{x})) \mid \ (\mathbf{x}) \}.\]
In this case, we get $\cdeg_{\C}^{\sharp}(\RR) = 0$ if and only if $\RR$ is Gorenstein.
}\end{Remark}

%%---------------
%\begin{Remark}{\rm Let $(\RR, \m)$ be a Cohen-Macaulay local ring  of dimension $d \geq 1$ that has a canonical ideal.
%A variation is to define a sharper (or strong) canonical degree by the formula
% \begin{eqnarray*}
%\cdeg_0(\RR)
%&=&   \sum_{\tiny \p \in \mbox{Min}(\C)} [\e_{0}(\C_{\p}) - \l((\RR/\C)_{\p})]\\ && \\
%&\geq& {\ds \sum_{\tiny \p \in \mbox{Min}(\C)} (\red(\C_{\p}) +1)\deg(\RR/\p) } \\ && \\
% & \geq & r+1  \\
% &=&   \sum_{\tiny \height \p=1} [\e_{0}(\C_{\p}) - \l
% ((\RR/\C)_{\p})] \deg(\RR/\p) .
%\end{eqnarray*}
%However, if $\dim \RR\geq 2$ and $\C$ is equimultiple, unlike $\cdeg(\RR)$, $\cdeg_0(\RR)$
%it is not expressed in terms of multiplicities.
%Observe that the same summation could be applied to a non-local ring as well.
%}\end{Remark}
%%--------------------

\end{document}